\documentclass[11pt,british]{amsart}

\usepackage{babel,eucal,url,amssymb,
enumerate,amscd,
}

\usepackage[pagebackref]{hyperref}

\theoremstyle{plain}
\newtheorem{thm}{Theorem}[section]
\newtheorem{lem}[thm]{Lemma}
\newtheorem{pro}[thm]{Proposition}

\theoremstyle{definition}
\newtheorem{defn}[thm]{Definition}

\theoremstyle{remark}
\newtheorem{rem}[thm]{Remark}
\newtheorem{example}[thm]{Example}

\newcommand{\Gtwo}{\ifmmode{{\rm G}_2}\else{${\rm G}_2$}\fi}

\date{\today}

\begin{document}

\title[]
 {Notes on a class of paracontact metric 3-manifolds}

\author[S. Zamkovoy]{Simeon Zamkovoy}

\address{
University of Sofia "St. Kl. Ohridski"\\
Faculty of Mathematics and Informatics\\
Blvd. James Bourchier 5\\
1164 Sofia, Bulgaria}
\email{zamkovoy@fmi.uni-sofia.bg}






\begin{abstract}
We study a class of 3-dimensional paracontact metric manifolds and we revise some of the results obtain in \cite{SS}.
\end{abstract}

\newcommand{\g}{\mathfrak{g}}
\newcommand{\s}{\mathfrak{S}}
\newcommand{\D}{\mathcal{D}}
\newcommand{\F}{\mathcal{F}}
\newcommand{\R}{\mathbb{R}}
\newcommand{\K}{\mathbb{K}}
\newcommand{\U}{\mathbb{U}}
\newcommand{\diag}{\mathrm{diag}}
\newcommand{\End}{\mathrm{End}}
\newcommand{\im}{\mathrm{Im}}
\newcommand{\id}{\mathrm{id}}
\newcommand{\Hom}{\mathrm{Hom}}

\newcommand{\Rad}{\mathrm{Rad}}
\newcommand{\rank}{\mathrm{rank}}
\newcommand{\const}{\mathrm{const}}
\newcommand{\tr}{{\rm tr}}
\newcommand{\ltr}{\mathrm{ltr}}
\newcommand{\codim}{\mathrm{codim}}
\newcommand{\Ker}{\mathrm{Ker}}

\newcommand{\thmref}[1]{Theorem~\ref{#1}}
\newcommand{\propref}[1]{Proposition~\ref{#1}}
\newcommand{\corref}[1]{Corollary~\ref{#1}}
\newcommand{\secref}[1]{\S\ref{#1}}
\newcommand{\lemref}[1]{Lemma~\ref{#1}}
\newcommand{\dfnref}[1]{Definition~\ref{#1}}


\newcommand{\ee}{\end{equation}}
\newcommand{\be}[1]{\begin{equation}\label{#1}}

\maketitle

\section{Introduction}\label{sec-1}
In these paper we consider the properties of the torsion tensor $\tau$ introduced in Section 3, and more specifically the tensor field $\nabla_{\xi}\tau$ which has an important role for the extremal metrics compact contact metric manifolds (see \cite{P}). We obtain various results which are analogous to ones for contact metric manifolds, but are somehow weaker because the metric tensor $g$ is indefinite. In Section 4, we consider a class of paracontact metric 3-manifolds for which the Ricci tensor $Q$ and the tensor field $\varphi$ commute ($Q\varphi=\varphi Q$). We revise the results obtain in \cite{SS} --- more specifically Theorem 3.1, which is both incorrectly formulated and proved, and Theorem 3.3, which is correctly formulated, but incorrectly proved. We given example illustrating our results. Besides, we find an invariant of the $\mathbb{D}$-homothetic transformation, as well as a necessary and sufficient condition for a 3-manifold to be $\emph{locally $\varphi-$}$ $\emph{symmetric}$.
\section{Preliminaries}\label{sec-2}
A (2n+1)-dimensional smooth manifold $M^{(2n+1)}$
has an \emph{almost paracontact structure}
$(\varphi,\xi,\eta)$
if it admits a tensor field
$\varphi$ of type $(1,1)$, a vector field $\xi$ and a 1-form
$\eta$ satisfying the  following compatibility conditions 
\begin{eqnarray}
  \label{f82}
    & &
    \begin{array}{cl}
     (i)   & \varphi(\xi)=0,\quad \eta \circ \varphi=0,\quad
     \\[5pt]
     (ii)  & \eta (\xi)=1 \quad \varphi^2 = id - \eta \otimes \xi,
     \\[5pt]
     (iii) & \textrm{distribution $\mathbb {D}: p \in M \longrightarrow \mathbb {D}_p\subset T_pM:$}
     \\[1pt]
     & \textrm{$\mathbb D_p=Ker \eta=\{X\in T_pM: \eta (X)=0\}$ is called {\it paracontact}}
     \\[1pt]
     & \textrm{{\it distribution} generated by $\eta$.}
    \end{array}
\end{eqnarray}

The tensor field $\varphi $ induces an almost paracomplex structure \cite{KW} on each
fibre on $\mathbb D$ and $(\mathbb D, \varphi , g_{\vert \mathbb D})$ is a $2n$-dimensional
almost paracomplex manifold. Since $g$ is non-degenerate metric on $M$ and $\xi $ is non-isotropic,
the paracontact distribution $\mathbb D$ is non-degenerate.

An immediate consequence of the definition of the almost
paracontact structure is that the endomorphism $\varphi$ has rank
$2n$,
$\varphi \xi=0$ and $\eta \circ \varphi=0$, (see \cite{B1,B2}
for the almost contact case).


If a manifold $M^{(2n+1)}$ with $(\varphi,\xi,\eta)$-structure
admits a pseudo-Riemannian metric $g$ such that
\begin{equation}\label{con}
g(\varphi X,\varphi Y)=-g(X,Y)+\eta (X)\eta (Y),
\end{equation}
then we say that $M^{(2n+1)}$ has an almost paracontact metric structure and
$g$ is called \emph{compatible}. Any compatible metric $g$ with a given almost paracontact
structure is necessarily of signature $(n+1,n)$.

Setting $Y=\xi$, we have
$\eta(X)=g(X,\xi).$

Any almost paracontact structure admits a compatible metric.
\begin{defn}
If $g(X,\varphi Y)=d\eta(X,Y)$ (where
$d\eta(X,Y)=\frac12(X\eta(Y)-Y\eta(X)-\eta([X,Y])$ then $\eta$ is
a paracontact form and the almost paracontact metric manifold
$(M,\varphi,\eta,g)$ is said to be a $\emph{paracontact metric
manifold}$.
\end{defn}

Denoting by $\pounds$ and $R$ the Lie differentiation and the curvature tensor respectively, we define the operators $\tau$, $l$ and $h$ by
\begin{equation}\label{1}
\tau=\pounds_{\xi}g, \quad lX=R(X,\xi)\xi, \quad h=\frac{1}{2}\pounds_{\xi}\varphi.
\end{equation}
The $(1,1)-$type tensors $h$ and $l$ are symmetric and satisfy
\begin{equation}\label{2}
l\xi=0,\quad h\xi=0,\quad trh=0,\quad trh\varphi=0 \quad and \quad h\varphi=-\varphi h.
\end{equation}
We also have the following formulas for a paracontact manifold:
\begin{equation}\label{3}
\nabla_{X}\xi=-\varphi X+\varphi hX \ (and \ hence \ \nabla_{\xi}\xi=0)
\end{equation}
\begin{equation}\label{4}
\nabla_{\xi}\varphi=0
\end{equation}
\begin{equation}\label{5}
trl=g(Q\xi,\xi)=-2n+trh^2
\end{equation}
\begin{equation}\label{6}
\varphi l\varphi+l=-2(\varphi^2-h^2)
\end{equation}
\begin{equation}\label{7}
\nabla_{\xi}h=-\varphi-\varphi l+\varphi h^2,
\end{equation}
where $tr$ is the trace of the operator, $Q$ is the Ricci operator and $\nabla$ is the Levi-Civita connection of $g$. The formulas are proved in \cite{Z}.

A paracontact metric manifolds for which $\xi$ is Killing is called a $K-\emph{paracontact}$ $\emph{manifold}$. A paracontact structure on $M^{(2n+1)}$ naturally gives rise to an almost paracomplex structure on the product $M^{(2n+1)}\times\Re$. If this almost paracomplex structure is integrable, the given paracontact metric manifold is said to be a $\emph{para-Sasakian}$. Equivalently, (see \cite{Z}) a paracontact metric manifold is a para-Sasakian if and only if
\begin{equation}\label{8}
(\nabla_{X}\varphi)Y=-g(X,Y)\xi+\eta(Y)X,
\end{equation}
for all vector fields $X$ and $Y$.

It is easy to see that a $3-$dimentional paracontact manifold is para-Sasakian if and only if $h=0$. For details we refer the reader to \cite{JW},\cite{Z}.

A paracontact metric structure is said to be  $\eta-\emph{Einstein}$ if
\begin{equation}\label{8.1}
Q=a.id+b.\eta\otimes\xi,
\end{equation}
where $a,b$ are smooth functions on $M^{(2n+1)}$. We also recall that the $k-$nullity distribution $N(k)$ of a pseudo-Riemannian manifold $(M,g)$, for a real number $k$, is the distribution
\begin{equation}
N_p(k)=\{Z\in T_pM:R(X,Y)Z=k(g((Y,Z)X-g(X,Z)Y)\},
\end{equation}
for any $X,Y\in T_pM$ ( see \cite{T}).

Finally, the sectional curvature $K(\xi,X)=\epsilon_XR(X,\xi,\xi,X)$, where $|X|=\epsilon_X=\pm 1$, of a plane section spanned by $\xi$ and the vector $X$ orthogonal to $\xi$ is called $\emph{$\xi$-sectional curvature}$, whereas the sectional curvature $K(X,\varphi X)=-R(X,\varphi X,\varphi X,X)$, where $|X|=-|\varphi X|=\pm 1$,
of a plane section spanned by vectors $X$ and $\varphi X$ orthogonal to $\xi$ is called a $\emph{$\varphi$-sectional curvature}$.

\section{Some properties of the torsion $\tau$}\label{sec-3}
In this chapter we discuss some aspects of the torsion of
paracontact metric manifolds. We begin with some preliminaries concerning
the tensor field $\tau$.
\begin{lem}
On a paracontact manifold $M^{(2n+1)}(\varphi,\xi,\eta,g)$ we have the formulas
\begin{equation}\label{f1}
\tau(X,Y)=-2g(\varphi X,hY)
\end{equation}
\begin{equation}\label{f2}
\tau(\xi,\cdot)=0
\end{equation}
\begin{equation}\label{f3}
\tau(X,Y)=\tau(Y,X)
\end{equation}
\begin{equation}\label{f4}
\tau(X,\varphi Y)=\tau(\varphi X,Y)
\end{equation}
\begin{equation}\label{f5}
\tau(\varphi X,\varphi Y)=\tau(X,Y)
\end{equation}
\end{lem}
\begin{proof}
Calculation is straightforward using \eqref{2} and \eqref{3}.
\end{proof}
\begin{pro}\label{p1}
Let $M^{(2n+1)}$ be a paracontact metric manifold with paracontact metric structure $(\varphi,\xi,\eta,g)$. Then the tensor field $\nabla_{\xi}\tau$ satisfies the following properties:

i) $(\nabla_{\xi}\tau)(X,Y)=(\nabla_{\xi}\tau)(Y,X)$

ii) $(\nabla_{\xi}\tau)(X,\cdot)=0$

iii) $(\nabla_{\xi}\tau)(\varphi X,\varphi Y)=(\nabla_{\xi}\tau)(X,Y)$

iv) for $X\in \mathbb {D}$, $|X|=\epsilon_X=\pm 1$, the sectional curvature $K(\xi,X)$ is given by
$$K(\xi,X)=-\frac{1}{2}\epsilon_X(\nabla_{\xi}\tau)(X,X)-1+\epsilon_X|hX|^2$$

v) $\nabla_{\xi}\tau=0$ if and only if $K(\xi,X)-K(\xi,Y)=\epsilon_X|hX|^2-\epsilon_Y|hY|^2$ for every $X,Y\in \mathbb {D}$, $|X|=\epsilon_X=\pm 1$ and $|Y|=\epsilon_Y=\pm 1$

vi) if $n=1$, then
$$-(\nabla_{\xi}\tau)(X,Y)=Ric(X,Y)+Ric(\varphi X,\varphi Y)-\eta(X)Ric(\xi,Y)-\eta(Y)Ric(\xi,X)+$$
$$+\eta(X)\eta(Y)Ric(\xi,\xi)$$
\end{pro}
\begin{proof}
i) $\nabla_{\xi}\tau$ is  symmetric, because $\tau$ is  symmetric.

ii) follows from \eqref{f2} and \eqref{3}.

iii) follows from \eqref{f4}, \eqref{f5} and \eqref{4}.

iv) From \eqref{f1} and \eqref{4} we obtain
\begin{equation}\label{f6}
(\nabla_{\xi}\tau)(X,Y)=-2g(\varphi X,(\nabla_{\xi}h)Y).
\end{equation}
From \eqref{4} and \eqref{f6}, since $h$ is a symmetric and anticommutes with $\varphi$, we get
$$K(\xi,X)=\epsilon_XR(X,\xi,\xi,X)=-\frac{1}{2}\epsilon_X(\nabla_{\xi}\tau)(X,X)-1+\epsilon_X|hX|^2.$$

v) If $\nabla_{\xi}\tau=0$, from iv), we have
$$K(\xi,X)-K(\xi,Y)=\epsilon_X|hX|^2-\epsilon_Y|hY|^2$$
for every $X,Y\in \mathbb {D}$, $|X|=\epsilon_X=\pm 1$ and $|Y|=\epsilon_Y=\pm 1$. Conversely, if the formula hods, then iv) implies $\epsilon_X(\nabla_{\xi}\tau)(X,X)=\epsilon_Y(\nabla_{\xi}\tau)(Y,Y)$.
Choosing $Y=\varphi X$, by iii) and $\epsilon_Y=-\epsilon_X$, we obtain
$$(\nabla_{\xi}\tau)(X,X)=0$$
for any $X,\in \mathbb {D}$, $|X|=\epsilon_X=\pm 1$. So, by ii), we have $\nabla_{\xi}\tau=0$.

vi) For $X\in \mathbb {D}$, $|X|=\epsilon_X=\pm 1$, since $h\varphi=-\varphi h$, we have $|hX|^2=-|h\varphi X|^2$ and hence iv) implies
\begin{equation}\label{f7}
K(\xi,X)-K(\xi,\varphi X)=-\epsilon_X(\nabla_{\xi}\tau)(X,X).
\end{equation}
Since $dimM=3$, from \eqref{f7} it follows that
$$Ric(X,X)+Ric(\varphi X,\varphi X)=-(\nabla_{\xi}\tau)(X,X).$$
Consequently, for $X,Y\in \mathbb {D}$ with $|X|=\epsilon_X=\pm 1$ and $|Y|=\epsilon_Y=\pm 1$,
$$Ric(X+Y,X+Y)+Ric(\varphi (X+Y),\varphi (X+Y))=-(\nabla_{\xi}\tau)(X+Y,X+Y)$$
implies
\begin{equation}\label{f8}
Ric(X,Y)+Ric(\varphi X,\varphi Y)=-(\nabla_{\xi}\tau)(X,Y).
\end{equation}
Finally for $X,Y\in TM$, $\varphi X,\varphi Y\in \mathbb {D}$  and $\varphi^2X=X-\eta(X)\xi$, $\varphi^2Y=Y-\eta(Y)\xi$, therefore by iii) and \eqref{f8}, we get the property vi).
\end{proof}
\begin{pro}\label{p1}
In a paracontact pseudo-Riemannian manifold $M^{(2n+1)}(\varphi,\xi,\eta,g)$, the following three conditions are equivalent:

i) $\nabla_{\xi}h=0$,

ii) $\nabla_{\xi}\tau=0$,

iii) $l\varphi=\varphi l$.
\end{pro}
\begin{proof}
From \eqref{f6} it follows that i) is equivalent to ii). Assuming i), from \eqref{7}
$$-l=\varphi^2-h^2.$$
Differentiating this equation with respect to $\xi$ and using \eqref{4}, we have
$$\nabla_{\xi}l=\nabla_{\xi}h^2=(\nabla_{\xi}h)h+h(\nabla_{\xi}h)=0.$$
Finally, from \eqref{6} and \eqref{7}, we obtain
\begin{equation}\label{f9}
2\nabla_{\xi}h=l\varphi-\varphi l
\end{equation}
and hence i) is equivalent to iii).
\end{proof}
\begin{rem}
These three conditions are equivalent to $\nabla_{\xi}l=0$ in the contact case. However, in the paracontact case from $\nabla_{\xi}l=0$ it follows only that $(\nabla_{\xi}h)^2=0$.
\end{rem}
\section{Main results}\label{sec-4}
In this section we consider  $3-$dimensional paracontact metric manifolds. Before we state our first result we need the following lemma which is incorrectly stated and proved in \cite{SS}. We will use many of the formulas which will appear in the proof.
\begin{lem}\label{l1}
Let $M^3$ be a paracontact metric manifold with a paracontact metric structure $(\varphi,\xi,\eta,g)$ such that $\varphi Q=Q\varphi$. Then the function $trl$ is constant everywhere on $M^3$.
\end{lem}
\begin{proof}
We recall that the curvature tensor of a 3-dimensional pseudo-Riemannian manifold is given by
\begin{equation}\label{9}
R(X,Y)Z=g(Y,Z)QX-g(X,Z)QY+g(QY,Z)X-g(QX,Z)Y-
\end{equation}
$-\frac{scal}{2}(g(Y,Z)X-g(X,Z)Y),$

where scal is the scalar curvature of the manifold.

Using $\varphi Q=Q\varphi$, \eqref{5} and $\varphi\xi=0$ we have that
\begin{equation}\label{10}
Q\xi=(trl)\xi.
\end{equation}
From \eqref{9} and using \eqref{1} and \eqref{10}, we have that for any $X$,
\begin{equation}\label{11}
lX=QX+(trl-\frac{scal}{2})X+\eta(X)(\frac{scal}{2}-2trl)\xi
\end{equation}
and hence $\varphi Q=Q\varphi$ and $\varphi\xi=0$ give
\begin{equation}\label{12}
\varphi l=l\varphi.
\end{equation}
By virtue of \eqref{12}, \eqref{6} and \eqref{7}, we obtain
\begin{equation}\label{13}
-l=\varphi^2-h^2
\end{equation}
and $\nabla_{\xi}h=0$. Differentiating \eqref{13} along $\xi$ and using \eqref{4} and $\nabla_{\xi}h=0$ we find that $\nabla_{\xi}l=0$ and therefore $\xi trl=0$. If at point $p\in M^3$ there exists
$X\in T_pM$, $X\neq\xi$ such that $lX=0$, then $l=0$ at the point $P$. In fact if $Y$ is the projection of $X$ on $\mathbb {D}$, we have $lY=0$, since $l\xi=0$. Using \eqref{12} we have $l\varphi Y=0$.
So $l=0$ at the point $P$ (and thus $trl=0$ at the point $P$). We now suppose that $l\neq 0$ on a neighborhood $U$ of the point $P$. Using \eqref{12} and that $\varphi$ is antisymmetric, we get $g(\varphi X,lX)=0$.
So $lX$ is parallel to $X$ for any $X$ orthogonal to $\xi$. It is not hard to see that $lX=\frac{trl}{2}X$ for any $X$ orthogonal to $\xi$. Thus for any $X$, we have
\begin{equation}\label{14}
lX=\frac{trl}{2}\varphi^2X
\end{equation}
Substituting \eqref{14} in \eqref{11} we get
\begin{equation}\label{15}
QX=aX+b\eta(X)\xi,
\end{equation}
where $a=\frac{scal-trl}{2}$ and $b=\frac{3trl-scal}{2}$. Differentiating \eqref{15} with respect to $Y$ and using \eqref{15} and $\nabla_{\xi}\xi=0$ we find
\begin{equation}\label{16}
(\nabla_{Y}Q)X=(Ya)X+((Yb)\eta(X)+bg(X,\nabla_Y\xi))\xi+b\eta(X)\nabla_Y\xi.
\end{equation}
So using $\xi trl=0$ and $\nabla_{\xi}\xi=0$, from  \eqref{16} with $X=Y=\xi$, we have $(\nabla_{\xi}Q)\xi=0$. Also using $h\varphi=-\varphi h$, and \eqref{3}, from \eqref{16} with $X=Y$ orthogonal to $\xi$, we get
\begin{equation}\label{17}
g((\nabla_X Q)X-(\nabla_{\varphi X})\varphi X,\xi)=0.
\end{equation}
But it is well known that $$(\nabla_XQ)X-(\nabla_{\varphi X})\varphi X+(\nabla_{\xi}Q)\xi=\frac{1}{2}grad(scal),$$
for any unit vector $X$ orthogonal to $\xi$. Hence, we easily get from the last two equations that $\xi(scal)=0$, and thus $\nabla_{\xi}Q=0$. Therefore, differentiating \eqref{9} with respect to $\xi$ and using $\nabla_{\xi}Q=0$,
we have $\nabla_{\xi}R=0$. So from the second identity of Bianchi, we get
\begin{equation}\label{18}
(\nabla_X R)(Y,\xi,Z)=(\nabla_{Y}R)(X,\xi,Z).
\end{equation}
Now, substituting \eqref{15} in \eqref{9}, we obtain
\begin{equation}\label{19}
R(X,Y)Z=(\gamma g(Y,Z)+b\eta(Y)\eta(Z))X-(\gamma g(X,Z)+b\eta(X)\eta(Z))Y+
\end{equation}
$$+b(\eta(X)g(Y,Z)-\eta(Y)g(X,Z))\xi,$$
where $\gamma=\frac{scal}{2}-trl$. For $Z=\xi$, \eqref{19} gives
\begin{equation}\label{20}
R(X,Y)\xi=\frac{trl}{2}(\eta(Y)X-\eta(X)Y).
\end{equation}
Using \eqref{20}, we obtain $(\nabla_X R)(Y,\xi,\xi)=\frac{X(trl)}{2}Y$, for $X,Y$ orthogonal to $\xi$. From this and \eqref{18} for $Z=\xi$, we get $(Xtrl)Y=(Ytrl)X$. Therefore $Xtrl=0$ for $X$ orthogonal to $\xi$, but $\xi(trl)=0$,
so the function $trl$ is constant and this completes the proof of the Lemma.
\end{proof}

\begin{rem}\label{r1}
When $l=0$ everywhere, then using \eqref{9}, \eqref{10} and \eqref{11} we get $R(X,Y)\xi=0$. This together with Theorem 3.3 in \cite{ZT} gives that $M^3$ is flat.
\end{rem}
From $Proposition~\ref{p1}$ and $Lemma~\ref{l1}$ we obtain the following
\begin{pro}
Let $M^3$ be a paracontact metric manifold with paracontact metric structure $(\varphi,\xi,\eta,g)$. If for every $X\in \mathbb {D}$, we have  $\varphi X\in \mathbb {D}$, then the conditions $Q\varphi=\varphi Q$ and
$\nabla_{\xi}\tau=0$ are equivalent.
\end{pro}
\begin{pro}\label{p2}
Let $M^3$ be a paracontact metric manifold with paracontact metric structure $(\varphi,\xi,\eta,g)$. Then the following conditions are equivalent:

i) $M^3$ is a $\eta-$Einstein

ii) $Q\varphi=\varphi Q$

iii) $\xi$ belongs to the $k-$nullity distribution
\end{pro}
\begin{proof}
$i\rightarrow ii$. This follows immediately from \eqref{8.1} and $\varphi\xi=0$.

$ii\rightarrow iii$. This follows immediately from \eqref{20} and $trl=const$.

$iii\rightarrow i$. By assumption, we have
\begin{equation}\label{21}
R(X,Y)\xi=k(\eta(Y)X-\eta(X)Y),
\end{equation}
where $k$ is a constant. From \eqref{21}, we have $Q\xi=2k\xi$ and so from \eqref{9}, we find
\begin{equation}\label{22}
R(X,Y)\xi=\eta(Y)QX-\eta(X)QY+(2k-\frac{scal}{2})(\eta(Y)X-\eta(X)Y).
\end{equation}
Comparing \eqref{21} and \eqref{22}, we get
\begin{equation}
\eta(Y)(QX+(k-\frac{scal}{2})X)-\eta(X)(QY+(k-\frac{scal}{2})Y)=0.
\end{equation}
Taking $Y$ orthogonal to $\xi$ and $X=\xi$, we have $QY=(\frac{scal}{2}-k)Y$ and so for any $Z$
\begin{equation}
QZ=(\frac{scal}{2}-k)Z+(3k-\frac{scal}{2})\eta(Z)\xi.
\end{equation}
This completes the proof.
\end{proof}
\begin{rem}
Because $a+b=trl$ (see formula \eqref{15}), using $Lemma~\ref{l1}$ and $Proposition~\ref{p1}$, we have the following: On any $\eta-$Einstein
$(Q=a.id+b.\eta\otimes\xi)$ paracontact metric manifold $M^3$, we have $a+b=const(=trl)$. It is known that for any $\eta-$Einstein para-Sasakian manifold
$M^{(2n+1)}$ $(n>1)$, we have $a=const.$ and $b=const.$ (see \cite{Z}).
\end{rem}
\begin{example}
Let $L$ be a 3-dimensional real connected Lie group and ${\g}$ be its Lie algebra with a basis $\{E_1,E_2,E_3\}$ of left invariant vector fields (see \cite{NZ}, \cite{ZN}). The Lie algebra ${\g}$ is determined
by the following commutators:
\begin{equation}\label {4.5}
\begin{array}{ll}
[E_1,E_2]=\alpha E_3, \quad [E_1,E_3]=\beta E_2, \\ \\

[E_2,E_3]=\beta E_1,  \quad \alpha\neq 0.
\end{array}
\end{equation}

We define an almost paracontact structure $(\varphi ,\xi ,\eta)$ and a pseudo-Riemannian metric $g$ in the following way:
\[
\begin{array}{llll}
\varphi E_1=E_2 , \quad \varphi E_2=E_1 , \quad \varphi E_3=0 \\
\xi =E_3 , \quad \eta (E_3)=1 , \quad \eta (E_1)=\eta (E_2)=0 , \\
g(E_1,E_1)=g(E_3,E_3)=-g(E_2,E_2)=1 ,\\
\quad g(E_i,E_j)=0, \quad i\neq j \in \{1,2,3\}.
\end{array}
\]
Then $(L,\varphi ,\xi ,\eta ,g)$ is a 3-dimensional almost paracontact metric manifold. So for $(L,\varphi ,\xi ,\eta ,g)$ to be a paracontact
metric structure we get $\alpha=-2$. Since the metric $g$ is left invariant the Koszul equality becomes
\begin{equation}
\begin{array}{l}
\nabla_{E_1}E_1=0,\quad \nabla_{E_1}E_2=-E_3,\quad \nabla_{E_1}E_3=-E_2\\
\nabla_{E_2}E_1=E_3,\quad \nabla_{E_2}E_2=0,\quad \nabla_{E_2}E_3=-E_1\\
\nabla_{E_3}E_1=-\frac{2\beta+2}{2}E_2,\quad \nabla_{E_3}E_2=-\frac{2\beta+2}{2}E_1,\quad \nabla_{E_3}E_3=0.
\end{array}
\end{equation}
The condition $Q\varphi=\varphi Q$ is equivalent to $\beta=0$. Now, it is not hard to see that
$$Ric(X,Y)=2g(X,Y)-4\eta(X)\eta(Y)\quad and \quad R(X,Y)\xi=-(\eta(Y)X-\eta(X)Y).$$
\end{example}
We have the following
\begin{pro}
Let $M^{(2n+1)}(\eta,\xi,\varphi,g)$ be a paracontact metric manifold. If $M^{(2n+1)}$ is an $\eta$-Einstein, the the Ricci tensor is given by
\begin{equation}\label{23}
Ric=(\frac{scal}{2n}+1+\frac{c^2}{4n})g+(-\frac{scal}{2n}-(2n+1)(1+\frac{c^2}{4n}))\eta\otimes\eta,
\end{equation}
where $c^2=\frac{1}{2}|\tau|^2$.

If, in addition $n=1$, then the curvature tensor is given by
\begin{equation}\label{24}
R(X,Y)Z=(\frac{scal}{2}+2(1+\frac{c^2}{4}))(g(Y,Z)X-g(X,Z)Y)+,
\end{equation}
$+(-\frac{scal}{2}-3(1+\frac{c^2}{4}))(\eta(Y)\eta(Z)X-\eta(X)\eta(Z)Y+g(Y,Z)\eta(X)\xi-g(X,Z)\eta(Y)\xi)$.
\end{pro}
\begin{proof}
Let $(e_i,\varphi e_i,\xi)$ be an orthonormal $\varphi$-bases. From equation \eqref{8.1} we have $Ric(\xi,\xi)=a+b$. Besides (see \cite{Z}) and the equality \eqref{5} we get
$$Ric(\xi,\xi)=-2n+trh^2=-2n+\frac{1}{4}|\pounds_{\xi}\varphi|^2=-2n-\frac{1}{4}|\pounds_{\xi}g|^2=-2n(1+\frac{c^2}{4n}).$$
Consequently
\begin{equation}\label{24.1}
a+b=-2n-\frac{c^2}{2}.
\end{equation}
Moreover \eqref{8.1} implies
\begin{equation}\label{24.2}
scal=(2n+1)a+b.
\end{equation}
From equalities \eqref{24.1} and \eqref{24.2} we get

$a=\frac{scal}{2n}+1+\frac{c^2}{4n}$ and $b=-\frac{scal}{2n}-(2n+1)(1+\frac{c^2}{4n})$.

Finally, when $n=1$, the curvature tensor is given by equality \eqref{9}. So \eqref{24}
follows from \eqref{23} and \eqref{9}.
\end{proof}
\begin{rem}
If $M^{(2n+1)}$ is an Einstein paracontact metric manifold, by \eqref{23} the scalar curvature $scal=-2n(2n+1)(1+\frac{c^2}{4n})$. If $M^{(2n+1)}$ is a $K-$paracontact Einstein manifold, then
$scal=-2n(2n+1)$. In the paracontact case from $scal=-2n(2n+1)$ follows only that $|\tau|^2=0$, but not that $M^{(2n+1)}$ is a $K-$paracontact Einstein manifold.
\end{rem}

Let $M^{(2n+1)}(\eta,\xi,\varphi,g)$ be a paracontact pseudo-Riemannian manifold and $\alpha$ a positive number. Making a change to the structure tensors of the form
$$\overline{\eta}=\alpha\eta, \overline{\xi}=\frac{1}{\alpha}\xi, \overline{\varphi}=\varphi, \overline{g}=\alpha g+\beta\eta\otimes\eta,$$
where $\beta=\alpha(\alpha-1)$. We get a new paracontact pseudo-Riemannian manifold  $M^{(2n+1)}(\overline{\eta},\overline{\xi},\overline{\varphi},\overline{g})$. This transformation is known as a
$\mathbb{D}$-homothetic transformation and it was introdused in \cite{Z}. By a direct computation one can see that the Ricci tensor transforms in the following manner:
$$\overline{Ric}(X,Y)=Ric(X,Y)+2\frac{\beta}{\alpha}g(X,Y)-2\frac{\beta}{\alpha^2}((2n+1)\alpha+n\beta)\eta(X)\eta(Y).$$
From the last equation, we obtain
$$\alpha\overline{Q}\xi=Q\xi-(1-\frac{1}{\alpha})(2n(\alpha+1)+trl)\xi.$$
Supposing that $M^{(2n+1)}(\eta,\xi,\varphi,g)$ satisfies $Q\xi=(trl)\xi$, we get from the last equation, $\overline{Q} \overline{\xi}=(tr\overline{l})\overline{\xi}$, where
$$tr\overline{l}=\frac{1}{\alpha^2}(trl-2n(\alpha^2-1)).$$
So we have proved the following
\begin{pro}
Let $M^{(2n+1)}(\eta,\xi,\varphi,g)$ be a paracontact pseudo-Riemannian manifold. Then the condition $Q\xi=(trl)\xi$ is invariant under a $\mathbb{D}$-homothetic transformation.
\end{pro}

\begin{defn}
A paracontact metric structure $(\varphi,\xi,\eta,g)$ is said to be $\emph{locally $\varphi-$}$ $\emph{symmetric}$ if $\varphi^2(\nabla_WR)(X,Y,Z)=0$, for all vector
fields $W,X,Y,Z$ orthogonal to $\xi$.
\end{defn}

We have the following
\begin{thm}
Let $M^3$ be a paracontact metric manifold with $Q\varphi=\varphi Q$. Then $M^3$ is locally $\varphi-$symmetric if and only if the scalar curvature scal of $M^3$ is constant.
\end{thm}
\begin{proof}
From the proof of $Lemma~\ref{l1}$, we see that either $l=0$ everywhere (and hence by $Remark~\ref{r1}$, that $M^3$ is a flat) or $trl=const\neq 0$ and in this case all the formulas in $Lemma~\ref{l1}$ are valid.
Differentiating \eqref{19} with respect to $W$ and using $Lemma~\ref{l1}$, we obtain
\begin{equation}\label{25}
(\nabla_WR)(X,Y,Z)=g(Y,Z)(W(b)\eta(X)\xi+b(g(X,\nabla_W\xi)\xi+\eta(X)\nabla_W\xi))-
\end{equation}
$$-g(X,Z)(W(b)\eta(Y)\xi+b(g(Y,\nabla_W\xi)\xi+\eta(Y)\nabla_W\xi))+$$
$$+(W(\gamma)g(\varphi^2Y,Z)+bg(g(Y,\nabla_W\xi)\xi+\eta(Y)\nabla_W\xi,Z))X-$$
$$-(W(\gamma)g(\varphi^2X,Z)+bg(g(X,\nabla_W\xi)\xi+\eta(X)\nabla_W\xi,Z))Y.$$
Taking $W,X,Y,Z$ orthogonal to $\xi$ and using $\varphi\xi=0$, we get the from \eqref{23}
$$2\varphi^2(\nabla_WR)(X,Y,Z)=W(scal)(g(Y,Z)X-g(X,Z)Y).$$
The rest of the proof follows immediately from this and $\xi(scal)$ (again see the proof of $Lemma~\ref{l1}$).
\end{proof}

\section*{Acknowledgments}

S.Z. is partially supported by Contract DFNI I02/4/12.12.2014 and Contract 80-10-33/2017 with the Sofia University "St.Kl.Ohridski".

\end{document}